\documentclass[11pt, reqno]{amsart}
\usepackage{amsmath}
\usepackage{amssymb}
\usepackage{amsfonts}
\usepackage{amssymb}
\usepackage{mathtools}
\usepackage[capitalize]{cleveref}
\usepackage{enumerate}
\usepackage[all]{xy}
\usepackage{url}

\theoremstyle{plain}
\newtheorem{thm}{Theorem}[section]

\newtheorem{prop}[thm]{Proposition}

\theoremstyle{definition}

\newcommand{\R}{\mathbb{R}}
\newcommand{\N}{\mathbb{N}}

\newcommand{\Z}{\mathbb{Z}}
\newcommand{\C}{\mathbb{C}}
\newcommand{\Ma}{\mathbb{M}}

\newcommand{\y}{\gamma}

\newcommand{\e}{\varepsilon}

\newcommand{\Ka}{\overline{K}^\text{alg}_1}

\newcommand{\minusa}{\mathrlap{\!\not{\phantom{\mathrm{a}}}}\mathrm{a}}

\title{Uniqueness of the $\zeta$ Transformation in Operator K-Theory}
\author{Mikkel Munkholm}

\begin{document}
\frenchspacing
    
\begin{abstract}
The classification of $^\ast$-homomorphisms between simple nuclear C$^\ast$-algebras led to the discovery of a new sequence of natural transformations $\zeta^n \colon K_0(\,\cdot\,; \mathbb Z/n\Z) \rightarrow \Ka$, with $n \geq 2$, between the operator $K_0$-group with coefficients in $\mathbb \Z/n\Z$ and the Hausdorffized unitary algebraic $K_1$-group.  In this paper, uniqueness of $\zeta^n$ up to compatibility with certain natural transformations $K_0(\,\cdot\,; \mathbb \Z/n\Z) \rightarrow K_1$ and $\Ka \rightarrow K_1$ is established.
\end{abstract}

\maketitle

\section{Introduction}
Classification problems have been a prevailing theme in operator algebras since the beginning of the subject. In the framework of C$^\ast$-algebras, Elliott's conjecture on the classification of simple nuclear C$^\ast$-algebras via $K$-theoretic and tracial data, cf. \cite{Elliott:ICM}, has been a major endeavor in the field over the last few decades. The collective effort of the C$^\ast$-algebra community culminated in the classification of unital separable simple nuclear C$^\ast$-algebras tensorially absorbing the Jiang--Su algebra $\mathcal{Z}$ from \cite{JiangSu} and which satisfy the Universal Coefficient Theorem of \cite{UCT}; see e.g. \cite{Kirchberg, Phillips, GLN1, GLN2, EGLN, TWW, CETWW} for an inexhaustive list. The final statement is recorded as Corollary~D in \cite{CETWW}. Further, Section~$1.2$ in \cite{Class} is recommended for a more detailed historical overview.

However, for applications, one often seeks to not only capture the isomorphism type of the C$^\ast$-algebras, but moreover the structure of the morphisms as well. For two C$^\ast$-algebras $A$ and $B$ in the aforementioned class, a complete description of the approximate unitary equivalence classes of the unital embeddings $A\rightarrow B$ was independently obtained in \cite{Class} and \cite{GLN:zeta}. More precisely, unital embeddings $\varphi, \psi \colon A \rightarrow B$ are approximately unitarily equivalent if and only if their $K$-theory $K_*$, $K_\ast(\, \cdot \, ; \Z/n\Z)$, $n \geq 2$, Hausdorffized unitary algebraic $K_1$-groups $\Ka$ and trace simplices agree.

The range of the invariant is more subtle. Computing the range amounts to identifying when a morphism on the invariant lifts to a $^\ast$-homomorphism between the involved C$^\ast$-algebras $A$ and $B$.  There are obstructions to lifting given by certain natural transformations connecting the components of the invariant.  In the traceless setting such as in \cite{Kirchberg, Phillips}, the morphisms between $K_*$ and $K_*(\,\cdot\,;\mathbb Z/n\Z)$ which admit lifts are precisely those compatible with the  \emph{Bockstein operatations}, see \cite{Bodig1,Bodig2}.

In addition to managing Bockstein operations, complications emerge whenever tracial data is present. In \cite{Thomsen2}, motivated by \cite{Det}, natural transformations
\begin{equation}\label{introSeq}
\xymatrix
{
	K_0 \ar[r]^-{\rho} & \mathrm{Aff}\,T \ar[r]^-{\mathrm{Th}}  & \Ka  \ar[r]^-{\minusa} & K_1
}
\end{equation}

\noindent are developed. We revisit this sequence in Section~\ref{Traces}. However, there are still morphisms on the level of the classifying invariant for which these natural transformations, alongside the Bockstein operations, are compatible and yet do not arise from $^\ast$-homomorphisms of the C$^\ast$-algebras. The missing ingredient is another collection of natural transformations $\zeta^n \colon K_0(\, \cdot \, ; \Z/n\Z) \rightarrow \Ka$ for $n\geq 2$ that binds the tracial data encoded in \eqref{introSeq} to $K_0(\, \cdot \, ; \Z/n\Z)$.  Compatibility with this sequence of natural transformations and the previously mentioned ones is sufficient to identify which morphisms of the invariant lift to $^\ast$-homomorphisms between classifiable C$^\ast$-algebras; the final statement is found in Theorem B in \cite{Class}.

These natural transformations $\zeta^n$ were introduced, independently, in \cite{GLN:zeta} and \cite{Class}. Their constructions are of very different natures. The one in \cite{GLN:zeta} was obtained abstractly via the existence of a (unnatural) splitting of the quotient map $\minusa_A \colon \Ka(A) \rightarrow K_1(A)$ for each $C^*$-algebra $A$, while the one of \cite{Class} was explicitly built using the de la Harpe--Skandalis determinant of \cite{Det}. We provide a recap on the construction in Section~\ref{Traces}.  A hands-on computation in the forthcoming paper \cite{Class2} will prove that these two natural transformations coincide. We demonstrate an alternative proof of this and, moreover, establish an abstract characterization of $\zeta^n$, yielding a uniqueness result for $\zeta^n$. In the following, $\nu^n_0 \colon K_0(\, \cdot \, ; \Z/n\Z) \rightarrow K_1$ denotes one of the aforementioned Bockstein operations.

\begin{thm}\label{MainThm}
For each $n\geq 2$, there exists a unique natural transformation $\zeta^n \colon K_0(\, \cdot \, ; \Z/n\Z) \rightarrow \Ka$ such that $\minusa \circ \zeta^n = \nu^n_0$.
\end{thm}

Despite providing an explicit formula for $\zeta$ in \cite{Class}, this approach carries the disadvantage of relying on a particular model for $K_\ast(\, \cdot \, ; \Z/n\Z)$. A strength of $K_\ast(\, \cdot \, ; \Z/n\Z)$ is its flexibility; $K_\ast(\, \cdot \, ; \Z/n\Z) \simeq K_\ast(\, \cdot \,  \otimes D_n)$ for any C$^\ast$-algebra $D_n$ satisfying the UCT and with $K_\ast(D_n) \simeq (0, \Z/n\Z)$, see \cite[Theorem~6.4]{Schochet}. Every model of $K_\ast(\, \cdot \, ; \Z/n\Z)$ has an attached family of Bockstein operations with an attached uniqueness property for the Bockstein operations, see e.g. \cite[Appendix~A]{Class}, allowing one to transition between models and maintain the compatibility with the inherent Bockstein operations. In the same spirit, the theorem above ensures that compatibility with $\zeta^n$ is independent of which model of $K(\,\cdot\,;\mathbb Z/n\mathbb Z)$ is used. 

In Sections~\ref{TotalK} and \ref{Traces}, we develop the apparatus entering in building the natural transformation $\zeta^n$, including $K$-theory with $\Z/n\Z$-coefficients and one of the Bockstein operations, tracial invariants and the sequence (\ref{introSeq}). During this, we describe how the sequence (\ref{introSeq}) may be adjusted to accommodate the non-unital setting. The final section is devoted to the proof of our abstract characterization of $\zeta^n$.

\subsection*{Acknowledgments} I am indebted to Christopher Schafhauser for his help and advise. Furthermore, I am grateful to the authors of \cite{Class} for sharing their preprint of their follow-up paper and permitting me to include observations on the non-unital Thomsen extension. Finally, this work was partially supported by the NSF grant DMS-$2400178$.

\section{K-theory with $\Z/n\Z$-coefficients}\label{TotalK}
We briefly recall the definition of $K$-theory with coefficients and Bockstein operations.  To this end, we define the \textit{dimension drop algebra}
\[
\mathbb{I}_n \coloneqq \lbrace f \in C([0,1], \Ma_n) : f(0) \in \C 1_{\Ma_n}, \, f(1) = 0 \rbrace,
\]
where $\Ma_n \coloneqq M_n(\C)$. With this definition, \textit{$K$-theory with $\Z/n\Z$-coefficients} of a C$^\ast$-algebra $A$ is
\[
K_\ast(A;\Z/n\Z) \coloneqq K_{\ast - 1}(A \otimes \mathbb{I}_n),
\]

\noindent see also \cite[Theorem~6.4]{Schochet}. For the Bockstein operations, to each C$^\ast$-algebra $A$, let $SA \coloneqq C_0((0,1), A)$ denote its suspension and consider the extension
\begin{equation}\label{Dimdrop}
\xymatrix{
0 \ar[r]  & A \otimes S\Ma_n  \ar[r]^-{\text{id}_A \otimes \iota}  &  A \otimes \mathbb{I}_n \ar[r]^-{\text{id}_A \otimes \e_0^n}  & A \ar[r]  & 0.
}
\end{equation}
Here, $\e_0^n \colon \mathbb I_n \rightarrow \mathbb C$ denotes the evaluation map at $t=0$ while $\iota \colon S\Ma_n \hookrightarrow \mathbb{I}_n$ denotes the inclusion. Due to $\mathbb{I}_n$ being nuclear, this sequence is indeed short-exact. The six-term exact sequence in $K$-theory takes the form
\begin{equation*}\label{SixTerm}
\xymatrix{
K_0(A) \ar[r]^-{\mu_{0,A}^n}  & K_0(A;\Z/n\Z) \ar[r]^-{\nu_{0,A}^n}  & K_1(A) \ar[d]^-{\times n} \\
K_0(A) \ar[u]^-{\times n}  & K_1(A;\Z/n\Z) \ar[l]_-{\nu_{1,A}^n}  & K_1(A) \ar[l]_-{\mu_{1,A}^n}
}
\end{equation*}
Here $\nu_{i,A}^n \coloneqq K_{i+1}(\text{id}_A \otimes \e_0^n)$ for $i = 0,1$ (mod $2$). The horizontal morphisms are some of the so-called \textit{Bockstein operations}. The naturality of the six-term exact sequence in $K$-theory implies that the Bockstein operations constitute natural transformations. A special case occurs when choosing $A=\C$ in \eqref{Dimdrop}, wherein one obtains a commutative diagram
\begin{equation*}\label{DimDropK}
\xymatrix{
\Z \ar[r]^-{\mu_{0,\C}^n}  & K_0(\C;\Z/n\Z) \ar[r]^-{\nu_{0,\C}^n}  & 0 \ar[d]^-{\times n} \\
\Z \ar[u]^-{\times n}  & K_1(\C;\Z/n\Z) \ar[l]_-{\nu_{1,\C}^n}  & 0 \ar[l]_-{\mu_{1,\C}^n}
}
\end{equation*}

\noindent Recalling that $K_\ast(\mathbb{I}_n) = K_{\ast-1}(\C, \Z/n\Z)$, one may deduce that $K_0(\mathbb{I}_n) \simeq 0$ and $K_1(\mathbb{I}_n) \simeq \Z/n\Z$. There are additional Bockstein operations, 
\[ 
    K_*(\,\cdot\,;\mathbb{Z}/n\mathbb Z) \rightarrow K_*(\,\cdot\,;\mathbb{Z}/mn\mathbb{Z}) \rightarrow K_*(\,\cdot\,;\mathbb Z/m\mathbb Z)
\]
for $m, n \geq 2$. We omit these as they will not be relevant here. K-theory with $\Z/n\Z$-coefficients always admits $n$-torsion. This may be recovered from \cite{Schochet}. We provide a slightly different, albeit short, proof.

\begin{prop}\label{Torsion}
For every integer $n\geq 2$ and any $\mathrm{C}^\ast$-algebra $A$, the groups $K_0(A; \Z/n\Z)$ and $K_1(A; \Z/n\Z)$ have $n$-torsion. 
\end{prop}

\begin{proof}
Due to (\ref{Dimdrop}) applied to the case $A=\C$, $\mathbb{I}_n$ is an extension of separable C$^\ast$-algebra satisfying the UCT, hence $\mathbb{I}_n$ satisfies the UCT by Proposition $2.4.7$ in \cite{RorClass}. Thus, the Künneth formula from \cite{UCT} applies.  We shall apply it in the form of Propositon $1.8$ in \cite{Schochet}. Since $K_0(\mathbb{I}_n)=0$ while $K_1(\mathbb{I}_n) = \Z/n\Z$, the Künneth sequence collapses to the $\Z/2\Z$-graded sequence
\[
\xymatrix{
0 \ar[r]  &  K_\ast(A) \otimes \Z/n\Z \ar[r]   &  K_{\ast}(A \otimes \mathbb{I}_n)  \ar[r]  & \mathrm{Tor}(K_{\ast+1}(A), \Z/n\Z) \ar[r] & 0.
}
\]
By Remark $7.11$ of \cite{UCT}, this sequence admits an unnatural splitting. Consequently,  $K_{\ast-1}(A;\Z/n\Z) \simeq (K_\ast(A)\otimes \Z/n\Z) \oplus \mathrm{Tor}_\Z^1(K_{\ast+1}(A),\Z/n\Z)$ unnaturally. The direct sum clearly has $n$-torsion, proving the claim.
\end{proof}

A second computation we shall rely on is surjectivity of the evaluation map in $K_\ast(\, \cdot \, ; \Z/n\Z)$. Observe that $K_\ast(\, \cdot \, ; \Z/n\Z)$ is functorial, since K-theory and the assignment of $A$ into (\ref{Dimdrop}) are both functorial. For notational convenience, we write $\e_A^n$ as a shorthand for the map $\text{id}_A \otimes \e_0^n$ from (\ref{Dimdrop}).

\begin{prop}\label{Surjective}
Let $n\geq 2$ be an integer and $A$ be some $\mathrm{C}^\ast$-algebra. Then the homomorphism $K_0(\e_A^n;\Z/n\Z) \colon  K_0(A \otimes \mathbb{I}_n;\Z/n\Z) \rightarrow K_0(A)$ is surjective.
\end{prop}

\begin{proof}
Let $C\Ma_n = C(0,1]\otimes \Ma_n$ denote the cone of $\Ma_n$ which contains $S\Ma_n$ as an ideal. The sequence (\ref{Dimdrop}) expands to the commutative diagram
\[
\xymatrix{
0 \ar[r]  & A \otimes S\Ma_n \ar[d]^-{\text{id}_{A}\otimes \mathrm{id}_{S\Ma_n}}  \ar[r]^-{\text{id}_A \otimes \iota}   &  A \otimes \mathbb{I}_n \ar[rr]^-{\e_A^n} \ar[d]^-{\text{id}_A \otimes \text{incl.}}  && A \ar[d]^-{\text{id}_A \otimes 1_{\Ma_n}} \ar[r]  & 0 \\
0 \ar[r]  & A \otimes S\Ma_n  \ar[r]  &  A \otimes C\Ma_n \ar[rr]_-{\text{id}_A \otimes \mathrm{ev}_0}  && A \otimes \Ma_n \ar[r]  & 0
}
\]
where $\iota$ is inclusion of $S\Ma_n$ into $\mathbb{I}_n$. By the six-term sequence and naturality of this sequence, we afford a commutative diagram
\[
\xymatrix{
K_0(A\otimes \mathbb{I}_n; \Z/n\Z) \ar[d]^-{K_0(\mathrm{id}_A \otimes \mathrm{incl.}; \Z/n\Z)}  \ar[rr]^-{K_0(\e_A^n; \Z/n\Z)}  && K_0(A;\Z/n\Z) \ar[d]^-{\times n} \ar[r]^-{\partial_0}  & K_1(A\otimes S\Ma_n) \ar[d]^-{\text{id}_{K_1(A\otimes S\Ma_n)}} \\
K_0(A\otimes C\Ma_n; \Z/n\Z) \ar[rr]  &&  K_0(A;\Z/n\Z) \ar[r]_-{\partial_0'}  & K_1(A\otimes S\Ma_n) 
}
\]
Here $\partial_0$ and $\partial_0'$ denote the attached exponential maps. Commutativity of the right-hand square in conjunction with \cref{Torsion} yields $\partial_0 = 0$. Due to exactness of the upper-row, this grants surjectivity of $K_0(\e_A^n; \Z/n\Z)$.
\end{proof}

\section{Traces and unitary algebraic $K_1$}\label{Traces}

This section develops the Thomsen sequence and thereby the target functor for the $\zeta$-transformation. We will extend the situation to the non-unital realm. Let $A$ and $B$ be some C$^\ast$-algebras, unital or not. A \textit{tracial functional} on $A$ will here refer to a linear functional $\tau \colon A \longrightarrow \C$ such that $\tau(ab) = \tau(ba)$ for each $a,b\in A$. The space of tracial states on $A$ is denoted by $T(A)$ while $T_{\leq 1}(A)$ comprises all contractive tracial functionals. By $\text{Aff }T(A)$ we denote the real vector space of continuous affine functions $f\colon T(A) \longrightarrow \R$. Note that $\text{Aff }T(\cdot)$ is functorial in the unital setting. That is, a unital $^\ast$-homomorphism $\varphi \colon A \longrightarrow B$ induces a positive unital linear map
\[
\text{Aff}\,T(\varphi) \colon \text{Aff}\,T(A) \longrightarrow \text{Aff}\, T(B), \, \text{Aff}\,T({\varphi})(f)(\tau) = f(\tau \circ \varphi).
\]
In this manner, the pairing between $K_0$ and $T(A)$ assumes the shape of the homomorphism
\[
\rho_A \colon K_0(A) \longrightarrow \text{Aff}\,T(A), \, \rho_A([p]_0 - [q]_0)(\tau) = \tau_n(p) - \tau_n(q),
\]
for projections $p,q\in \Ma_n \otimes A$, where $\tau_n$ is the unnormalized extension of $\tau$ and $\text{Aff}\, T(\cdot)$ is viewed as an abelian group.

The tracial data of $\text{Aff}\, T(A)$ and the attached pairing map $\rho_A$ as described admits a connection to the unitary algebraic $K_1$-group. Unitary algebraic $K_1$ in the unital setting traces back to \cite{Thomsen1} and \cite{Thomsen2}. We refer to Section $2.2$ in \cite{Class} for a more in depth account.

Let $A$ be a unital C$^\ast$-algebra and denote by $DU_n(A)$ the derived subgroup of the group $U_n(A)$ of unitaries in $\text{M}_n(A)$, i.e., the subgroup generated by all its commutators. The canonical unital inclusion of $\text{M}_n(A)$ into $\text{M}_{n+1}(A)$ preserves the commutators, hence induces an inductive limit $U_\infty(A)$ equipped with the inductive limit topology. With this in mind, the \textit{Hausdorffized unitary algebraic $K_1$-group} is then the quotient
$$
\Ka(A) := U_\infty(A)/\overline{DU_\infty(A)}.
$$

A class in $\Ka(A)$ is denoted by $[u]_\text{alg}$. The link to traces is rectified via the Thomsen sequence from \cite{Thomsen2}; see e.g. Proposition $2.9$(i) of \cite{Class} for a proof alongside its definition. It was successfully employed in classification of $A\mathbb{T}$ and $AH$-algebras, cf. \cite{Thomsen1}. The \textit{Thomsen sequence} is the sequence
\begin{equation}\label{ThomsenExt1}
\xymatrix{
K_0(A) \ar[r]^-{\rho_A}  & \text{Aff}\, T(A) \ar[r]^-{\text{Th}_A}  & \Ka(A) \ar[r]^-{\minusa_A} & K_1(A) \ar[r] & 0.
}
\end{equation}
Here $\minusa_A([u]_\text{alg}) = [u]_1$ and $\text{Th}(\mathrm{ev}_h) = [e^{2\pi i h}]$ for some self-adjoint $h\in A$ and with $\mathrm{ev}_h\in \text{Aff}\, T(A)$ being the evaluation, i.e. $\mathrm{ev}_h(\tau)=\tau(h)$; the choice of $h$ is irrelevant according to the discussion following the proof of Proposition $2.10$ in \cite{Class}. As part of the work in the same section of \cite{Class}, continuity of $\mathrm{Th}_A$ is established. The sequence is exact, except for the detail that 
\begin{equation}\label{NearExact}
\ker \mathrm{Th}_A = \overline{\mathrm{im }\rho_A}.
\end{equation}

\noindent For this near-exactness property we recommend Proposition $2.9$ in \cite{Class} for a detailed account.

In a forthcoming paper \cite{Class2}, the non-unital situation is addressed. We shall reiterate the ideas in \cref{NonutailSection}. For now, unitary algebraic $K_1$ is functorial in the unital case: given a unital $^\ast$-homomorphism $\varphi \colon A \longrightarrow B$, there exists a homomorphism $\Ka(\varphi) \colon \Ka(A) \longrightarrow \Ka(B)$ such that
\begin{equation}\label{Thomsen}
\xymatrix
{
K_0(A) \ar[d]_-{K_0(\varphi)} \ar[r]^-{\rho_A} & \text{Aff }T(A) \ar[r]^-{\text{Th}_A} \ar[d]_-{\text{Aff}\, T(\varphi)}  & \Ka(A) \ar[r]^-{\minusa_A} \ar[d]_-{\Ka(\varphi)}   & K_1(A) \ar[d]_-{K_1(\varphi)} \\
K_0(B) \ar[r]^-{\rho_B} & \text{Aff}\, T(B) \ar[r]^-{\text{Th}_B}  & \Ka(B) \ar[r]^-{\minusa_B}   & K_1(B).
}
\end{equation}
commutes; see Proposition $2.10$ in \cite{Class}. This further entails that the Thomsen sequence (\ref{ThomsenExt1}) may be regarded as a functor in the variable $A$ with attached natural transformations $\rho$, $\mathrm{Th}$ and $\minusa$.

\subsection{One model of $\zeta^n$}
The $\zeta$-transformation will be introduced. The model exhibited here is found in section $3.1$ of \cite{Class}, and we point the reader there for details on all the involved claims. Let $A$ and $B$ be some unital C$^\ast$-algebras. For each piecewise smooth path $u \colon [0,1] \longrightarrow U_n(B)$, the \textit{de la Harpe -- Skandalis determinant} $\Delta_B(u)$ of $u$ is the element in $\text{Aff}\, T(B)$ given by
\[
\Delta_B(u)(\tau) = \frac{1}{2\pi i} \int_{[0,1]} \tau_n ( u'(t) u(t)^\ast) \ dt, \quad \tau \in T(B).
\]

\noindent Recall that $\tau_n$ refers to the unnormalized extension of $\tau$ to $\text{M}_n(B)$. The determinant defines a continuous group homomorphism from $U_\infty(C([0,1],B))$ into $\mathrm{Aff}T(B)$. Depending solely on the homotopy class, the map $u \mapsto \Delta_B(u)$ descends to a homomorphism $\det_B \colon U_\infty^0(B) \longrightarrow \text{Aff}\, T(B)/\overline{\text{im }\rho}$. The determinant was examined in \cite{Det} and moreover used to deduce the near-exactness \eqref{NearExact} property of the Thomsen sequence in \cite{Thomsen2}. 

To each $n\geq 2$, $(A\otimes \mathbb{I}_n)^\sim$ may be identified with the $\mathrm{C}^\ast$-algebra continuous functions $f\colon [0,1] \rightarrow A\otimes \mathbb{M}_n$ satisfying $f(0)\in A\otimes 1_{\mathbb{M}_n}$ and $f(1) \in \C 1_{A\otimes \Ma_n}$. Therefore, every $u\in U_\infty((A\otimes \mathbb{I}_n)^\sim)$ may be viewed as being an element of $U_\infty(C([0,1],A))$. Consequently, for each $n\geq 2$, one has a homomorphism $\widetilde{\zeta}_{A}^n \colon U_\infty((A\otimes \mathbb{I}_n)^\sim) \longrightarrow \Ka(A)$ via
\begin{equation}\label{zeta}
\widetilde{\zeta^n}(u) = \big [ \text{ev}_A^{(0,n)}(u) \big ]_\text{alg} - \big [ \text{ev}_A^{(1,n)} \big ]_\text{alg} + \text{Th}_A \big ( n^{-1} \Delta_{A}(u) \big ).
\end{equation}

Here $\text{ev}_A^{(i,n)}$ is evaluation at $i$ from $(A\otimes \mathbb{I}_n)^\sim$ into $A$ for $i=0,1$. This gives rise to a natural transformation $A \mapsto \zeta_A^n$ between the functors $\Ka$ and $K_\ast(\, \cdot \, , \Z/n\Z)$ such that $\nu_{0}^n = \minusa \circ \zeta^n$ for all $n\geq 1$. This is the \textit{$\zeta$-transformation}. The naturality of $\zeta^n$ stems from naturality of $\Delta$ and the evaluation maps.

The $\zeta^n$-tranformation from \cite{GLN:zeta} emerged from a rather different approach although with the same feature that $\nu_0^n = \minusa \circ \zeta^n$ for all $n\geq 1$. Our proof is a modification of their argument to include an abstract and unique characterization. However, in order to demonstrate this, we must extend the Thomsen sequence to the non-unital setting.

\subsection{Non-unital case}\label{NonutailSection}
Suppose $B$ is some non-unital C$^\ast$-algebra. Following the same ideas as when defining $K_1$-groups in the non-unital setting, to force split-exactness one reinvents the Hausdorffized unitary algebraic $K_1$-group as the kernel of the canonical character on the unitisation. As such, we define $\Ka(B)$ to be the kernel of canonical character $\pi_B \colon B^\sim \rightarrow \C$, so that it fits into the short exact sequence
\[
\xymatrix{
0 \ar[r] & \Ka(B) \ar[r]^-{\iota^\text{alg}}   & \Ka(B^\sim) \ar[rr]^-{\Ka(\pi_B)}  && \Ka(\C) \ar[r] & 0.
}
\]

\noindent We modify (\ref{ThomsenExt1}) to encompass all the non-unital C$^\ast$-algebras. Accordingly, let
\[
\text{Aff}_0\,  T_{\leq 1}(B) = \big \lbrace f \colon T_{\leq 1}(B) \longrightarrow \R :  \text{$f$ is affine, continuous and } f(0) = 0  \big \rbrace.
\]

\noindent Now, each $\tau \in T_{\leq 1}(B)$ canonically extends to a positive tracial functional $\tau^\sim$ on $B^\sim$ via 
\[
\tau^\sim(x + \lambda 1_{B^\sim}) = \tau(x) + \lambda, \quad x\in B, \ \lambda \in \C,
\]
which remains contractive. Then the natural maps
\begin{align*}
\eta \colon \ker \text{Aff} \, (\pi_B) \longrightarrow \text{Aff}_0 \, T_{\leq 1}(B), \quad \eta(f)(\tau) = f(\tau^\sim), \\
\mu \colon \text{Aff}_0 \, T_{\leq 1}(B) \longrightarrow \text{Aff }T(B^\sim), \quad \mu(f)(\tau) = f(\tau_{\vert B}),
\end{align*}
are affinely homeomorphic mutual inverses as one readily verifies via direct computations. Upon using the identification $\ker \text{Aff} \, (\pi_B) \simeq \text{Aff}_0\,  T_{\leq 1}(B)$ and (\ref{Thomsen}) from the unital case, we obtain the commutative diagram found below:
\begin{equation*}
\xymatrix{
& 0 \ar[d] &   0 \ar[d] &    0 \ar[d] &    0 \ar[d]  & \\
& K_0(B)   \ar[d]_-{\text{incl.}}  \ar@{.>}[r]^-{\rho_B}    & \text{Aff}_{0}\, T_{\leq 1}(B) \ar@{.>}[r]^-{\text{Th}_{B}} \ar[d]_-{\text{incl.}}    & \Ka(B) \ar@{.>}[r]^-{\minusa_B}   \ar[d]_-{\text{incl.}}   & K_1(B) \ar[d]_-{\text{incl.}}  \ar[r] & 0  \\
& K_0(B^\sim) \ar[d]_-{K_0(\pi_B)}  \ar[r]^-{\rho_{B^\sim}}  & \text{Aff }T(B^\sim) \ar[d]_-{ \text{Aff} \, (\pi_B)}  \ar[r]^-{\text{Th}_{B^\sim}}   & \Ka(B^\sim)  \ar[d]_-{\Ka(\pi_B)} \ar[r]^-{\minusa_{B^\sim}}  & K_1(B^\sim) \ar[d]_-{K_1(\pi_B)} \ar[r] & 0 \\
0\ar[r] & K_0(\C) \ar[r]^-{\rho_\C}  \ar[d]  &  \text{Aff }T(\C) \ar[r]^-{\text{Th}_\C} \ar[d]  & \Ka(\C) \ar[r]^-{\minusa_\C}  \ar[d] & K_1(\C) \ar[r] \ar[d] & 0   \\
& 0& 0 & 0 & 0 &
}
\end{equation*}
The dashed morphisms on the second row are defined to be restrictions of the third row ones. The middle row with unitisations satisfies that
\begin{equation}\label{diagram}
\ker \mathrm{Th}_{B^\sim} = \overline{\mathrm{im } \rho_{B^\sim}} \ \text{ and } \ 
\mathrm{im } \mathrm{Th}_{B^\sim} = \ker \minusa_{B^\sim},
\end{equation}

\noindent since $B^\sim$ is unital, see \cref{Traces}.

We now demonstrate that \eqref{diagram} does hold for $B$, meaning $\ker \minusa_B = \text{im }\text{Th}_B$ and $\overline{\text{im }\rho_B} = \ker \text{Th}_B$. The containment  $\overline{\mathrm{im } \rho_B} \subseteq \ker \mathrm{Th}_B$ is a direct consequence of continuity of the Thomsen map $\mathrm{Th}_{B^\sim}$ and a rudimentary computations shows that $\minusa_B \circ \mathrm{Th}_B = 0$. 

For the containment $\ker \mathrm{Th}_B \subseteq \overline{\mathrm{im } \rho_B}$, suppose $\mathrm{Th}_B(f) = 0$ holds for some $f\in \mathrm{Aff}_0 \, T_{\leq 1}(B)$. Regarded as an element in $\text{Aff} \, T(B^\sim)$, the property \eqref{diagram} of the middle row yields $f = \lim_n \rho_{B^\sim}(x_n)$ for a sequence $(x_n)_{n\geq 1} \subseteq K_0(B^\sim)$. Since $f\in \ker  \text{Aff} \, (\pi_B) \simeq \text{Aff}_0 \, T_{\leq 1}(B)$,
\[
\rho_\C \big ( \lim_n K_0(\pi_B)(x_n) \big ) = 
\lim_n  \text{Aff} \, (\pi_B) \rho_{B^\sim}(x_n) = 
 \text{Aff} \, (\pi_B) (f) = 0.
\]

\noindent However, $\rho_\C$ is injective, whereby $K_0(\pi_B)(x_n) =0$ for infinitely many $n\in \N$. Upon passing to a subsequence if necessary, we may write $f = \lim_n \rho_{B^\sim}(x_n)$ with $x_n \in \ker K_0(\pi_B) = K_0(B)$ for each $n\geq 1$. Thus, $f\in \overline{\mathrm{im } \rho_B}$.

For the containment $\ker \minusa_B \subseteq \mathrm{im } \mathrm{Th}_B$, suppose $x\in \ker \minusa_B$. If we now view $x$ as an element of $\Ka(B^\sim)$, then $\minusa_{B^\sim}(x) = 0$. From right-exactness of the middle row, it follows that $x = \mathrm{Th}_{B^\sim}(f)$ for an $f\in \mathrm{Aff} \, T(B^\sim)$. Put
\[
g = f -  \text{Aff} \, (\pi_B) \cdot 1,
\]

\noindent which does belong to $\mathrm{Aff} \, T(B^\sim)$ via $ \text{Aff} \, (\pi_B)(f) \in \mathrm{Aff} \, T(\C) \simeq \R$. In fact, we may infer that
\[
g = f -  \text{Aff} \, (\pi_B)(f) \cdot 1 \in \ker  \text{Aff} \, (\pi_B) = \mathrm{Aff}_0 \, T_{\leq 1}(B).
\]

\noindent Recall that for $h\in \mathrm{Aff}_0 \, T_{\leq 1}(B)$ one has $\mathrm{Th}_{B^\sim}(h) = [e^{2\pi i a}]_{\mathrm{alg}}$ with $a\in A_{\mathrm{sa}}$ satisfying $h = \mathrm{ev}_a$. In particular, we must have $\mathrm{Th}_{B^\sim}(1) = 0$, hence
\[
\mathrm{Th}_{B}(g) = \mathrm{Th}_{B^\sim}(g) 
                   = x - \big [ e^{2\pi i  \text{Aff} \, (\pi_B)(f)} \big ]_{\mathrm{alg}}
                   = x.
\]

\noindent This completes the proof.

\section{The Main Theorem}\label{Proof}

In this section, we demonstrate the existence and the uniqueness of the $\zeta$-invariant. The uniqueness portion is the new contribution. Nevertheless, we deduce existence for completeness. The argument hinges on the observation below, which furthermore carries the existence part of \cref{MainThm}.  It deviates from (\ref{zeta}) by giving an abstract picture of $\zeta$.

\begin{thm}\label{Existence}
For each integer $n\geq 2$, there exists a unique natural transformation
\[
\lambda^n \colon K_0(\, \cdot \, ; \Z/n\Z) \longrightarrow \Ka(\, \cdot \, \otimes \mathbb{I}_n)
\]
satisfying $\minusa_{A\otimes \mathbb{I}_n} \circ \lambda_A^n = \mathrm{id}_{K_1(A;\Z/n\Z)}$ for every $\mathrm{C}^\ast$-algebra $A$ and every $n\geq 2$. In particular, there exists a natural transformation
\[
\zeta_{}^n \colon K_0(\, \cdot \, ; \Z/n\Z) \longrightarrow \Ka
\]
 satisfying $\minusa_A \circ \zeta_A^n = \nu_A^n$ for each $n\geq 2$ and every $\mathrm{C}^\ast$-algebra $A$.
\end{thm}

\begin{proof}
Let $A$ be some C$^\ast$-algebra and set $B = A \otimes \mathbb{I}_n$. Firstly, we observe that $K_0(B)$ has $n$-torsion via \cref{Torsion} while $\mathrm{Aff}_0 \, T_{\leq 1}(B)$ is a vector space, hence is torsion-free. Therefore, the pairing map $\rho_B \colon K_0(B) \longrightarrow \mathrm{Aff}_0 \, T_{\leq 1}(B)$ vanishes. Consequently, this in conjunction with \eqref{diagram}, having $B^\sim$ replaced by $B$, entails that Thomsen sequence collapses to the extension
\begin{equation}\label{finalseq}
\xymatrix{
0 \ar[r]  &  \mathrm{Aff}_0 \, T_{\leq 1}(B) \ar[r]^-{\mathrm{Th}_B}  & \Ka(B) \ar[r]^-{\minusa_B}  & K_1(B) \ar[r] & 0.
}
\end{equation}

\noindent Being a vector space, $\mathrm{Aff}_0 \, T_{\leq 1}(B)$ is injective as a module. Injective modules enable the use of the standard splitting lemma. Therefore, we obtain some splitting $\lambda_B \colon K_1(B) \longrightarrow \Ka(B)$ for $\minusa_B$. For uniqueness, suppose $\lambda_B'$ were another splitting for $\minusa_B$. Set $\delta \coloneqq \lambda_B - \lambda_B'$. By construction, $\minusa_B$ vanishes on the image of $\delta$, so the image of $\delta$ falls into $\ker \minusa_B = \mathrm{Im}\, \mathrm{Th}_B \simeq \mathrm{Aff}_0 \, T_{\leq 1}(B)$. As such, there exists a (unique) homomorphism $\theta_B \colon K_1(B) \longrightarrow \mathrm{Aff}_0 \, T_{\leq 1}(B)$ for which $\delta = \theta_B \circ \minusa_B$. Since the codomain is torsion-free whereas the domain has $n$-torsion, $\theta_B = 0$. It follows that $\delta=0$ and thus $\lambda_B=\lambda_{B'}$ as required. This proves unqiueness of the splitting.


Naturality may be obtained in a similar manner. Indeed, let $\varphi \colon D \longrightarrow E$ be a $^\ast$-homomorphism between C$^\ast$-algebras. The induced map $\varphi \otimes \text{id}_{\mathbb{I}_n}$ gives rise to a group homomorphism $\pi_\varphi \colon K_0(D;\Z/n\Z) \rightarrow \Ka(E)$ by
\[
\pi_\varphi \coloneqq \lambda_E \circ K_1(\varphi\otimes \mathrm{id}_{\mathbb{I}_n}) - \Ka(\varphi \otimes \text{id}_{\mathbb{I}_n}) \circ \lambda_D.
\]
Now, $D\mapsto \lambda_D$ is natural if $\pi_\varphi = 0$ for each such $\varphi$. By naturality of $A\mapsto \minusa_{A}$,
\begin{align*}
\minusa_E \circ \pi_\varphi &= K_1(\varphi \otimes \mathrm{id}_{\mathbb{I}_n}) - \minusa_E \circ \Ka(\varphi \otimes \mathrm{id}) \circ \lambda_D  \\
&=  K_1(\varphi \otimes \mathrm{id}_{\mathbb{I}_n}) - K_1(\varphi \otimes \mathrm{id}_{\mathbb{I}_n}) \circ \minusa_D \circ \lambda_D = 0.
\end{align*}

Thus, $\pi_\varphi$ factors through the torsion-free group $\mathrm{Aff}_0 \, T_{\leq 1}(E)$. With exactly the same argument as in the preceding paragraph, one arrives at $\pi_\varphi = 0$ due to $K_1(D\otimes \mathbb{I}_n)$ having $n$-torsion via \cref{Torsion}.

For the existence of $\zeta$, let $A$ be any C$^\ast$-algebra and fix some $n\geq 2$. Define accordingly a group homomorphism by
\[
\zeta_A^n \coloneqq \Ka(\e_A^n) \circ \lambda_{A\otimes \mathbb{I}_n} \colon K_0(A;\Z/n\Z) \longrightarrow \Ka(A).
\]

\noindent Due to $A\mapsto \Ka(\e_{A}^n)$ and $A\mapsto \lambda_{A}$ being natural, the assignment $A \mapsto \zeta_A^n$ is natural for each integer $n\geq 2$. Lastly, since $A\mapsto \minusa_{A}$ is natural,
\[
\minusa_A \circ \zeta_A^n = 
\minusa_A \circ \Ka(\e_A^n) \circ \lambda_{A\otimes \mathbb{I}_n} = 
K_1(\e_A^n) \circ \minusa_{A\otimes \mathbb{I}_n} \circ \lambda_{A\otimes \mathbb{I}_n} = 
K_1(\e_A^n) = \nu_A^n
\]
 holds for each integer $n\geq 2$ and C$^\ast$-algebra $A$.
\end{proof}

We proceed to deriving \cref{MainThm} by supplying the missing uniqueness aspect of  the preceding theorem. The proof reuses arguments from the existence proof with modifications to pass from the torsion case to the general setting.

\begin{thm}
Let $n\geq 2$ be any integer. If $\y^n, \theta^n \colon K_0(\, \cdot \, ; \Z/n\Z) \longrightarrow \Ka$ are natural transformations such that $\minusa_A \circ \y_A^n = \minusa_A \circ \theta_A^n$ holds for each $\mathrm{C}^\ast$-algebras $A$, then $\y^n = \theta^n$. Moreover, such a natural transformation exists.
\end{thm}

\begin{proof}
We initially reapply the uniqueness argument for the splitting found in \cref{Existence}. Let $A$ be a C$^\ast$-algebra and set $B = A\otimes \mathbb{I}_n$. Due to \cref{Torsion}, the pairing map $\rho_B \colon K_0(B) \longrightarrow \mathrm{Aff}_0 \, T_{\leq 1}(B)$ vanishes. As such, the Thomsen sequence collapses to \eqref{finalseq}. By the hypothesis on $\y^n$ and $\theta^n$, the attached error $\delta \coloneqq \y_B^n - \theta_B^n$ has its image contained in $\ker \minusa_B \simeq \mathrm{Aff}_0 \, T_{\leq 1}(B)$. Thus, $\delta = \mathrm{Th}_B \circ \theta$ for a homomorphism $\theta \colon K_0(B;\Z/n\Z) \longrightarrow \mathrm{Aff}_0 \, T_{\leq 1}(B)$. Since the domain of $\theta$ has $n$-torsion whereas its codomain is torsion-free, it must vanish and so $\delta=0$. In total, one must have $\y_{A\otimes \mathbb{I}_n}^n = \theta_{A\otimes \mathbb{I}_n}^n$.

Upon appealing to naturality of the involved transformations with respect to the evaluation map $\e_A^n \colon A\otimes \mathbb{I}_n \rightarrow A$, one obtains a commutative square
\[
\xymatrix{
K_0(A\otimes \mathbb{I}_n; \Z/n\Z) \ar[rr]^-{K_0(\e_A^n;\Z/n\Z)} \ar[d]_-{\y_{A\otimes \mathbb{I}_n}^n = \theta_{A\otimes \mathbb{I}_n}^n}  && K_0(A;\Z/n\Z) \ar[d]^-{\y_A^n} \\
\Ka(A\otimes \mathbb{I}_n) \ar[rr]_-{\Ka(\e_A^n)}  && \Ka(A)
}
\]
There is another similar diagram with the right-vertical morphism being $\theta_A^n$ instead of $\y_A^n$. Subsequently, we arrive at
\[
\theta_A^n \circ K_0(\e_A^n; \Z/n\Z) = \y_A^n \circ K_0(\e_A^n;\Z/n\Z).
\]
Due to \cref{Surjective}, $K_0(\e_A^n;\Z/n\Z)$ is surjective, hence has a right inverse. It follows that $\theta_A^n = \y_A^n$. The existence part is handled in \cref{Existence}.
\end{proof}

The theorem does not rely on the model chosen for $K_\ast(\, \cdot \, ; \Z/n\Z)$. In fact, the proof solely used the conclusion of \cref{Torsion} and \cref{Surjective}, each of which may be deduced without appealing to our selected picture of $K_\ast(\, \cdot \, ; \Z/n\Z)$.  Ergo, any natural transformation satisfying the conditions of the theorem must agree with the one of (\ref{zeta}) on page $6$.

\bibliographystyle{plain}
\bibliography{zeta}

@misc{Class,
      title={Classifying $^*$-homomorphisms {I}: {U}nital simple nuclear {C$^*$}-algebras}, 
      author={José R. Carrión and James Gabe and Christopher Schafhauser and Aaron Tikuisis and Stuart White},
      eprint={2307.06480},
      archivePrefix={arXiv},
      primaryClass={math.OA},
      note={arXiv:2307.064802},
      url={https://arxiv.org/abs/2307.06480}, 
}

@inproceedings {Elliott:ICM,
    AUTHOR = {Elliott, George A.},
     TITLE = {The classification problem for amenable {$\mathrm{C}^\ast$}-algebras},
 BOOKTITLE = {Proceedings of the {I}nternational {C}ongress of
              {M}athematicians, {V}ol. 1, 2 ({Z}\"{u}rich, 1994)},
     PAGES = {922--932},
 PUBLISHER = {Birkh\"{a}user, Basel},
      YEAR = {1995},
   MRCLASS = {46L05 (46L35 46L80 46M15)},
  MRNUMBER = {1403992},
MRREVIEWER = {Robert S. Doran},
}

@misc{Class2,
	title={Classifying $^*$-homomorphisms {II}}, 
	author={José R. Carrión and James Gabe and Christopher Schafhauser and Aaron Tikuisis and Stuart White},
	year={2023},
	eprint={2307.06480},
	archivePrefix={arXiv},
	primaryClass={math.OA},
	note={in preparation},
	url={https://arxiv.org/abs/2307.06480}, 
}

@article{Phillips,
	AUTHOR = {Phillips, N. Christopher},
	TITLE = {A classification theorem for nuclear purely infinite simple
	{$\mathrm{C}^\ast$}-algebras},
	JOURNAL = {Doc. Math.},
	FJOURNAL = {Documenta Mathematica},
	VOLUME = {5},
	YEAR = {2000},
	PAGES = {49--114},
	ISSN = {1431-0635},
	MRCLASS = {46L05 (19K56 46L35 46L80)},
	MRNUMBER = {1745197},
	MRREVIEWER = {Mikael R\o rdam},
}

@article{Kirchberg,
	AUTHOR={Kirchberg, Eberhard},
	TITLE={The classification of purely infinite {$\mathrm{C}^\ast$}-algebras using {K}asparov's theory},
	NOTE={Manuscript available at \url{https://ivv5hpp.uni-muenster.de/u/echters/ekneu1.pdf}}
}

@article {JiangSu,
    AUTHOR = {Jiang, Xinhui and Su, Hongbing},
     TITLE = {On a simple unital projectionless {$\mathrm{C}^\ast$}-algebra},
   JOURNAL = {Amer. J. Math.},
  FJOURNAL = {American Journal of Mathematics},
    VOLUME = {121},
      YEAR = {1999},
    NUMBER = {2},
     PAGES = {359--413},
      ISSN = {0002-9327,1080-6377},
   MRCLASS = {46L35 (19K35 46L80)},
  MRNUMBER = {1680321},
MRREVIEWER = {Vicumpriya\ S.\ Perera},
       URL = {http://muse.jhu.edu/journals/american_journal_of_mathematics/v121/121.2jiang.pdf},
}

@article{EGLN,
    AUTHOR = {Elliott, George A. and Gong, Guihua and Lin, Huaxin and Niu,
              Zhuang},
     TITLE = {On the classification of simple amenable {$\mathrm{C}^\ast$}-algebras with finite decomposition rank, {II}},
   JOURNAL = {J. Noncommut. Geom.},
  FJOURNAL = {Journal of Noncommutative Geometry},
    VOLUME = {19},
      YEAR = {2025},
    NUMBER = {1},
     PAGES = {73--104},
      ISSN = {1661-6952},
   MRCLASS = {46L35 (46L05)},
  MRNUMBER = {4860189},
       DOI = {10.4171/jncg/560},
       URL = {https://doi.org/10.4171/jncg/560},
}

@article{Schochet,
	AUTHOR = {Schochet, Claude},
	TITLE = {Topological methods for {$\mathrm{C}^\ast$}-algebras. {IV}. {M}od
	{$p$} homology},
	JOURNAL = {Pacific J. Math.},
	FJOURNAL = {Pacific Journal of Mathematics},
	VOLUME = {114},
	YEAR = {1984},
	NUMBER = {2},
	PAGES = {447--468},
	ISSN = {0030-8730},
	MRCLASS = {46L80 (19K33 46M20 55N99)},
	MRNUMBER = {757511},
	MRREVIEWER = {Vern Paulsen},
	URL = {http://projecteuclid.org.libproxy.unl.edu/euclid.pjm/1102708718},
}

@article {UCT,
    AUTHOR = {Rosenberg, Jonathan and Schochet, Claude},
     TITLE = {The {K}\"unneth theorem and the universal coefficient theorem
              for {K}asparov's generalized {$K$}-functor},
   JOURNAL = {Duke Math. J.},
  FJOURNAL = {Duke Mathematical Journal},
    VOLUME = {55},
      YEAR = {1987},
    NUMBER = {2},
     PAGES = {431--474},
      ISSN = {0012-7094,1547-7398},
   MRCLASS = {46L80 (19K33 46M20 58G12)},
  MRNUMBER = {894590},
MRREVIEWER = {Thierry\ Fack},
       DOI = {10.1215/S0012-7094-87-05524-4},
       URL = {https://doi.org/10.1215/S0012-7094-87-05524-4},
}

@article {CETWW,
    AUTHOR = {Castillejos, Jorge and Evington, Samuel and Tikuisis, Aaron
              and White, Stuart and Winter, Wilhelm},
     TITLE = {Nuclear dimension of simple {$\mathrm{C}^\ast$}-algebras},
   JOURNAL = {Invent. Math.},
  FJOURNAL = {Inventiones Mathematicae},
    VOLUME = {224},
      YEAR = {2021},
    NUMBER = {1},
     PAGES = {245--290},
      ISSN = {0020-9910,1432-1297},
   MRCLASS = {46L35 (46L05)},
  MRNUMBER = {4228503},
MRREVIEWER = {Changguo\ Wei},
       DOI = {10.1007/s00222-020-01013-1},
       URL = {https://doi.org/10.1007/s00222-020-01013-1},
}

@article {TWW,
    AUTHOR = {Tikuisis, Aaron and White, Stuart and Winter, Wilhelm},
     TITLE = {Quasidiagonality of nuclear {$\mathrm{C}^\ast$}-algebras},
   JOURNAL = {Ann. of Math. (2)},
  FJOURNAL = {Annals of Mathematics. Second Series},
    VOLUME = {185},
      YEAR = {2017},
    NUMBER = {1},
     PAGES = {229--284},
      ISSN = {0003-486X,1939-8980},
   MRCLASS = {46L05 (47L40)},
  MRNUMBER = {3583354},
MRREVIEWER = {Dinesh\ Jayantilal\ Karia},
       DOI = {10.4007/annals.2017.185.1.4},
       URL = {https://doi.org/10.4007/annals.2017.185.1.4},
}

@article {GLN2,
    AUTHOR = {Gong, Guihua and Lin, Huaxin and Niu, Zhuang},
     TITLE = {A classification of finite simple amenable {$\mathcal Z$}-stable
              {$\mathrm{C}^\ast$}-algebras, {II}: {$\mathrm{C}^\ast$}-algebras
              with rational generalized tracial rank one},
   JOURNAL = {C. R. Math. Acad. Sci. Soc. R. Can.},
  FJOURNAL = {Comptes Rendus Math\'{e}matiques de l'Acad\'{e}mie des Sciences. La
              Soci\'{e}t\'{e} Royale du Canada. Mathematical Reports of the Academy
              of Science. The Royal Society of Canada},
    VOLUME = {42},
      YEAR = {2020},
    NUMBER = {4},
     PAGES = {451--539},
      ISSN = {0706-1994},
   MRCLASS = {46L35 (46L05 46L80)},
  MRNUMBER = {4215380},
}

@article {GLN1,
    AUTHOR = {Gong, Guihua and Lin, Huaxin and Niu, Zhuang},
     TITLE = {A classification of finite simple amenable {$\mathcal Z$}-stable
              {$\mathrm{C}^\ast$}-algebras, {I}: {$\mathrm{C}^\ast$}-algebras with generalized
              tracial rank one},
   JOURNAL = {C. R. Math. Acad. Sci. Soc. R. Can.},
  FJOURNAL = {Comptes Rendus Math\'{e}matiques de l'Acad\'{e}mie des Sciences. La
              Soci\'{e}t\'{e} Royale du Canada. Mathematical Reports of the Academy
              of Science. The Royal Society of Canada},
    VOLUME = {42},
      YEAR = {2020},
    NUMBER = {3},
     PAGES = {63--450},
      ISSN = {0706-1994},
   MRCLASS = {46L35 (46L05 46L80)},
  MRNUMBER = {4215379},
}

@article {GLN:zeta,
    AUTHOR = {Gong, Guihua and Lin, Huaxin and Niu, Zhuang},
     TITLE = {Homomorphisms into simple {$\mathcal Z$}-stable {$\mathrm{C}^\ast$}-algebras,
              {II}},
   JOURNAL = {J. Noncommut. Geom.},
  FJOURNAL = {Journal of Noncommutative Geometry},
    VOLUME = {17},
      YEAR = {2023},
    NUMBER = {3},
     PAGES = {835--898},
      ISSN = {1661-6952,1661-6960},
   MRCLASS = {46L35 (46L05)},
  MRNUMBER = {4627097},
       DOI = {10.4171/jncg/490},
       URL = {https://doi.org/10.4171/jncg/490},
}

@incollection {RorClass,
    AUTHOR = {R{\o}rdam, M.},
     TITLE = {Classification of nuclear, simple {$\mathrm{C}^\ast$}-algebras},
 BOOKTITLE = {Classification of nuclear {$\mathrm{C}^\ast$}-algebras. {E}ntropy in
              operator algebras},
    SERIES = {Encyclopaedia Math. Sci.},
    VOLUME = {126},
     PAGES = {1--145},
 PUBLISHER = {Springer, Berlin},
      YEAR = {2002},
      ISBN = {3-540-42305-X},
   MRCLASS = {46L05 (19K56 46L35 46L80)},
  MRNUMBER = {1878882},
MRREVIEWER = {Judith\ A.\ Packer},
       DOI = {10.1007/978-3-662-04825-2\_1},
       URL = {https://doi.org/10.1007/978-3-662-04825-2_1},
}

@article {Thomsen1,
    AUTHOR = {Nielsen, Karen Egede and Thomsen, Klaus},
     TITLE = {Limits of circle algebras},
   JOURNAL = {Exposition. Math.},
  FJOURNAL = {Expositiones Mathematicae. International Journal},
    VOLUME = {14},
      YEAR = {1996},
    NUMBER = {1},
     PAGES = {17--56},
      ISSN = {0723-0869},
   MRCLASS = {46L80 (19K14 46L05)},
  MRNUMBER = {1382013},
MRREVIEWER = {Sze-Kai\ Tsui},
}

@article {Thomsen2,
    AUTHOR = {Thomsen, Klaus},
     TITLE = {Traces, unitary characters and crossed products by {$\mathbb
              Z$}},
   JOURNAL = {Publ. Res. Inst. Math. Sci.},
  FJOURNAL = {Kyoto University. Research Institute for Mathematical
              Sciences. Publications},
    VOLUME = {31},
      YEAR = {1995},
    NUMBER = {6},
     PAGES = {1011--1029},
      ISSN = {0034-5318,1663-4926},
   MRCLASS = {46L05 (46L80)},
  MRNUMBER = {1382564},
MRREVIEWER = {Kevin\ McClanahan},
       DOI = {10.2977/prims/1195163594},
       URL = {https://doi.org/10.2977/prims/1195163594},
}

@article {Det,
    AUTHOR = {de la Harpe, P. and Skandalis, G.},
     TITLE = {D\'eterminant associ\'e{} \`a{} une trace sur une alg\'ebre de
              {B}anach},
   JOURNAL = {Ann. Inst. Fourier (Grenoble)},
  FJOURNAL = {Universit\'e{} de Grenoble. Annales de l'Institut Fourier},
    VOLUME = {34},
      YEAR = {1984},
    NUMBER = {1},
     PAGES = {241--260},
      ISSN = {0373-0956,1777-5310},
   MRCLASS = {46L80 (18F25 19K14 19K56 46L05 58G12)},
  MRNUMBER = {743629},
MRREVIEWER = {G.\ A.\ Elliott},
       DOI = {10.5802/aif.958},
       URL = {https://doi.org/10.5802/aif.958},
}

@article {Bodig1,
    AUTHOR = {B\"odigheimer, Carl-Friedrich},
     TITLE = {Splitting the {K}\"unneth sequence in {$K$}-theory. {II}},
   JOURNAL = {Math. Ann.},
  FJOURNAL = {Mathematische Annalen},
    VOLUME = {251},
      YEAR = {1980},
    NUMBER = {3},
     PAGES = {249--252},
      ISSN = {0025-5831,1432-1807},
   MRCLASS = {55N15 (55U25)},
  MRNUMBER = {589253},
MRREVIEWER = {Jerome\ Kaminker},
       DOI = {10.1007/BF01428944},
       URL = {https://doi.org/10.1007/BF01428944},
}

@article {Bodig2,
    AUTHOR = {B\"odigheimer, Carl-Friedrich},
     TITLE = {Splitting the {K}\"unneth sequence in {$K$}-theory},
   JOURNAL = {Math. Ann.},
  FJOURNAL = {Mathematische Annalen},
    VOLUME = {242},
      YEAR = {1979},
    NUMBER = {2},
     PAGES = {159--171},
      ISSN = {0025-5831,1432-1807},
   MRCLASS = {55N15 (18G15 46M20)},
  MRNUMBER = {537958},
MRREVIEWER = {Jerome\ Kaminker},
       DOI = {10.1007/BF01420413},
       URL = {https://doi.org/10.1007/BF01420413},
}

\end{document}